\documentclass[12pt,reqno]{amsart}
\usepackage[pagebackref,colorlinks=true]{hyperref}
\usepackage{float}
\usepackage{changepage}
\usepackage{verbatim}
\usepackage{amsfonts,amsmath,amsthm,amssymb,amscd,enumerate,eucal,url}
\usepackage{mathtools}
\usepackage[margin=1in]{geometry}
\usepackage{xcolor}
\usepackage{tikz}
\usepackage{microtype}
\usetikzlibrary{arrows.meta,positioning}

\usepackage{mathrsfs}

\numberwithin{equation}{section}
\newlength{\extralength}
\newcommand{\R}{\mathbb{R}}

\newcommand{\Rp}{\R_{>0}}

\newcommand{\tr}{\operatorname{tr}}
\newcommand{\im}{\operatorname{im}}

\newtheorem{Theorem}{Theorem}[section]
\newtheorem{Lemma}{Lemma}[section]
\newtheorem{Proposition}{Proposition}[section]
\newtheorem{Corollary}{Corollary}[section]
\newtheorem{Definition}{Definition}[section]
\newtheorem{Remark}{Remark}[section]
\newtheorem{Example}{Example}[section]

\newtheoremstyle{axiomstyle}
  {}{}                 % razmaci
  {\itshape}           % telo
  {}                   % indent
  {\bfseries}          % naslov
  {.}                  % tačka iza broja
  {0.5em}              % razmak
  {\thmname{#1}~\thmnumber{#2}\thmnote{ (#3)}} % OVO PRIKAZUJE (RG0...)
  
\theoremstyle{axiomstyle}

\title{Golden and Metallic Structures on Hessian Manifolds}
\date{\today}

\author{Jonathan Washburn}
\address{Jonathan Washburn\\
Recognition Physics Institute, Austin, Texas, USA}
\email{jon@recognitionphysics.org}

\author{Milan Zlatanovi\'c}
\address{Milan Zlatanovi\'c\\
Department of Mathematics, Faculty of Science and Mathematics, University of Ni\v s, Vi\v segradska 33, 18000 Ni\v s, Serbia}
\email{zlatmilan@yahoo.com}

\begin{document}

\newcommand{\add}[1]{{\color{blue}#1}}
\newcommand{\del}[1]{{\color{red}#1}}
\newenvironment{added}{\begingroup\color{blue}}{\endgroup}
\newenvironment{deleted}{%
  \begingroup\color{red}%
  % Prevent deleted text from affecting citation order/numbering.
  \renewcommand{\cite}[2][]{\relax}%
  % Prevent deleted text from affecting labels/counters/structure.
  \renewcommand{\label}[1]{\relax}%
  \let\ref\relax
  \let\eqref\relax
  \let\pageref\relax
  \renewcommand{\section}[1]{\par\medskip\noindent{\bfseries [deleted section] ##1}\par}%
  \renewcommand{\subsection}[1]{\par\medskip\noindent{\bfseries [deleted subsection] ##1}\par}%
  \renewcommand{\subsubsection}[1]{\par\medskip\noindent{\bfseries [deleted subsubsection] ##1}\par}%
  \renewcommand{\paragraph}[1]{\par\noindent{\bfseries [deleted paragraph] ##1}\par}%
  \renewcommand{\subparagraph}[1]{\par\noindent{\bfseries [deleted subparagraph] ##1}\par}%
  \renewenvironment{equation}{\[\ignorespaces}{\]\ignorespacesafterend}%
  \renewenvironment{theorem}[1][]{\par\medskip\noindent\textbf{[deleted theorem] }\ignorespaces}{\par\medskip}%
  \renewenvironment{lemma}[1][]{\par\medskip\noindent\textbf{[deleted lemma] }\ignorespaces}{\par\medskip}%
  \renewenvironment{proposition}[1][]{\par\medskip\noindent\textbf{[deleted proposition] }\ignorespaces}{\par\medskip}%
  \renewenvironment{corollary}[1][]{\par\medskip\noindent\textbf{[deleted corollary] }\ignorespaces}{\par\medskip}%
  \renewenvironment{definition}[1][]{\par\medskip\noindent\textbf{[deleted definition] }\ignorespaces}{\par\medskip}%
  \renewenvironment{remark}[1][]{\par\medskip\noindent\textbf{[deleted remark] }\ignorespaces}{\par\medskip}%
  \renewenvironment{example}[1][]{\par\medskip\noindent\textbf{[deleted example] }\ignorespaces}{\par\medskip}%
  \renewenvironment{notation}[1][]{\par\medskip\noindent\textbf{[deleted notation] }\ignorespaces}{\par\medskip}%
  \renewenvironment{convention}[1][]{\par\medskip\noindent\textbf{[deleted convention] }\ignorespaces}{\par\medskip}%
  \renewenvironment{axiom}[1][]{\par\medskip\noindent\textbf{[deleted axiom] }\ignorespaces}{\par\medskip}%
}{\endgroup}

\newcommand{\Poly}{\R[u,v]}

 \begin{abstract}We consider the reciprocal cost function
\(
J(x)=\frac12(x+x^{-1})-1
\)
and its $n$-dimensional {extension} %MDPI: Independent Equation is not allowed in an abstract, only text is allowed, so we changed into inline equation, please confirm this revision. %MZ it is okay.
\(J(x_1,\ldots,x_n)
=
\frac12(R+R^{-1})-1,
R=\prod\limits_{i=1}^n x_i^{\alpha_i},
\alpha=(\alpha_1,\ldots,\alpha_n)\in\mathbb{R}^n\setminus\{0\}.\)
In logarithmic coordinates \(t_i=\log x_i\), the Hessian of \(J\) has a rank of one at every point. The associated Hessian geometry is degenerate and does not define a Riemannian metric.
To obtain a nondegenerate geometric structure, we introduce a family of Hessian metrics \(h_\lambda\).  Combining the rank-one tensor with the Hessian metric \(h_\lambda\), we construct a \mbox{\((1,1)\)-tensor} field \(A_\lambda\). Its trace normalization defines a projector \(P_\lambda\), which induces an almost product structure and the corresponding golden and metallic structures. 
We study several  properties of the projector \(P_\lambda\)
and the induced structures, including eigendistributions,
parallelism, integrability, and curvature. The construction is given in an arbitrary dimension, and explicit formulas are obtained in the two-dimensional case. In particular, we show that the projector \(P_\lambda\) is generally not parallel with respect to either the canonical flat affine connection or the Levi-Civita connection \(\nabla^\lambda\) of the Hessian metric \(h_\lambda\).
\bigskip


\medskip 


\end{abstract}
\maketitle

\setcounter{tocdepth}{3}

%\tableofcontents

\newcommand{\config}{\mathcal{C}}
\newcommand{\configR}{\mathcal{C}_R}

\section{Motivation}

The golden ratio has been known since Euclid and appears under different names, such as the golden section, divine ratio, golden mean, or golden proportion. It occurs in nature, especially in patterns related to Fibonacci numbers, like phyllotaxis and certain~flowers.

It also appears in music, in~harmonic relations, and~in proportions of the human body. From~ancient times, it has played an important role in architecture and art, for~example in the proportions of temples, sculptures, and~paintings. The~golden ratio can be defined geometrically by dividing a segment into two parts such that the ratio of the whole to the larger part equals the ratio of the larger part to the smaller one. This ratio is the positive solution of the equation 
$x^2-x-1=0.$ 
It appears  in geometric figures such as the pentagon, decagon, and dodecagon. On~the other hand, let us consider the general quadratic equation
\[
x^2 - \alpha x - \beta = 0,
\]
where $\alpha$ and $\beta$ are positive integers. Its positive solution is
\[
\sigma_{\alpha,\beta} = \frac{\alpha + \sqrt{\alpha^2 + 4\beta}}{2},
\]
which defines the \emph{{metallic means family}}. %MDPI: Please confirm if the italics are necessary; if not, please remove them. The following highlights are the same. %MZ it is okay
 This family includes, for~instance, the~golden mean, the~silver mean, the~subtle mean, etc., and~was introduced by {Spinadel} %MDPI: We rearranged all the references citations to appear in numerical order. Please confirm this revision.
 \cite{spinadel1997, spinadel1999}. %MZ it is okay

{\color{black}The metallic means arise as limiting ratios of generalized secondary Fibonacci sequences (GSFSs)} (see~\cite{spinadel2000, stakhov2007})  
\[
G(n+1)=pG(n)+qG(n-1), \qquad n\ge 1,
\]
with $G(0)=a\in \mathbb{R}$, $G(1)=b\in \mathbb{R}$ and $p,q\in\mathbb{R}$. The~ratio ${G(n+1)}/{G(n)}$ of two consecutive terms of GSFSs converges~to 
\begin{itemize}
\item The golden mean 
\(
\varphi=\frac{1+\sqrt{5}}{2},
\)
for $p=q=1$, determined by the ratio of two consecutive classical Fibonacci numbers;

\item The silver mean 
\(
\sigma_{2,1}=1+\sqrt{2},
\)
for $p=2$ and $q=1$, determined by the ratio of two consecutive Pell numbers;

\item The bronze mean 
\(
\sigma_{3,1}=\frac{3+\sqrt{13}}{2},
\)
for $p=3$ and $q=1$;

\item The subtle mean 
\(
\sigma_{4,1}=2+\sqrt{5}=\varphi^3,
\)
for $p=4$ and $q=1$;

\item The copper mean 
\(
\sigma_{1,2}=2,
\)
for $p=1$ and $q=2$;

\item The nickel mean 
\(
\sigma_{1,3}=\frac{1+\sqrt{13}}{2},
\)
for $p=1$ and $q=3$.
\end{itemize}

{In} %MDPI: We added indention for this paragraph. Please confirm this revision. %MZ it is okay
 the case $q=1$ and $p=k$, one gets the $k$-Fibonacci sequence
\[
F_{k,n+1}=kF_{k,n}+F_{k,n-1}, \qquad F_{k,0}=0,\; F_{k,1}=1,
\]
which generalizes the classical Fibonacci sequence. {\color{black}These metallic means will appear later as the eigenvalues of the cost-induced metallic operators \(M^\lambda_{p,q}\).}

  The golden ratio appears in  quasicrystals, dynamical systems, and~certain models in mathematical physics (see, e.g.,~\cite{Coldea,Livio, Penrose, Steinhardt} and references therein).

Following~\cite{WZ,WZ1}, we consider the canonical reciprocal cost function in one {dimension} %MDPI: (1). The italics of variables with the same meaning should be consistent in the full text. Please check the full text and modify. (2). Please carefully check to make sure there are no Duplicate Equations throughout the text, thank you. %MZ it is okay. thanks. 
\begin{equation}\label{eq:Jintro}
J(x) = \frac{1}{2}\!\left(x + x^{-1}\right) - 1,\qquad x>0.
\end{equation}
 {Cost} %MDPI: Please confirm if paragraphs without indentation below equations should be retained, please check all. %MZ There is no need change. 
 functions are ubiquitous in optimization problems, and~different cost functions can have different motivations. In~\cite{WZ}, it is proven that this particular function appears as a unique solution of the polynomial composition law together with the curvature calibration.  {\color{black}{This identifies \eqref{eq:Jintro} as the canonical reciprocal cost. The~quadratic equation \(x^{2}=x+1\), which defines the golden ratio, is the simplest case of the polynomial composition law, while the curvature calibration determines the second-order behaviour of \(J\) at its minimum. These properties motivate the study of the geometric structures induced by the reciprocal \mbox{cost function}.}}

 For \(\alpha=(\alpha_1,\dots,\alpha_n)\in\R^n\setminus\{0\}\) the \(n\)-dimensional extension obtained by composing \eqref{eq:Jintro} with  \(R=\prod_i x_i^{\alpha_i}\) is
\[
J(x_1,\dots,x_n)=\frac12(R+R^{-1})-1,
\]
{so that \(J(t)=\cosh(\alpha\cdot t)-1\) in logarithmic coordinates. The~Hessian of \(J\) is then the {\color{black} rank-one tensor
\(
\tilde g:=\nabla^2 J=\cosh(\alpha\cdot t)\,\alpha\otimes\alpha,
\)}
which is positive semidefinite of rank one and  does not define a Riemannian metric. {\color{black}The rank-one property is the starting point of the construction. It determines a single distinguished direction generated by the gradient of \(\log R\), from~which the projector and the induced almost product, golden, and~metallic structures are obtained.}

Let \(M\) be a smooth manifold and \(I\) the identity endomorphism of the tangent bundle \(TM\). A~\((1,1)\)-tensor field \(Q\) on \(M\) is called a \emph{{polynomial structure}} %MZ it is okay
if it satisfies a polynomial identity \(P(Q)=0\). The~two quadratic cases studied in this paper are the \emph{{golden structure}} 
%MZ it is okay

~\vspace{-12pt}\begin{equation}Q^2=Q+I,\end{equation} which is motivated by the classical golden ratio equation. More generally, for~positive integers \(p,q\), the~{\it \((p,q)\)-{metallic structure}} %MZ it is okay, and also is okay for the rest italic text ine the paper
(see, e.g.,~\cite{hretcanu2013}) is defined by
\begin{equation}\label{eq:metallic-intro}
Q^2 = p\,Q + q\,I.
\end{equation}

Both golden and metallic structures belong to a broader class of polynomial structures introduced by Goldberg and Yano~\cite{goldberg}. Golden structures on differentiable manifolds were first introduced by Hrețcanu and Crășmareanu~\cite{hretcanu2007}, who further developed their properties in~\cite{crasmareanu2008}. Using an approach similar to the one developed for golden structures, Hrețcanu and 
Crășmareanu studied the  metallic structures on Riemannian manifolds in~\cite{hretcanu2013}. {\color{black}Since then, golden and metallic structures have become an active research topic in differential geometry, with~a comprehensive recent survey~\cite{chen} and many papers devoted to new classes of these structures, their integrability,  and~applications on Riemannian manifolds (see, e.g.,~\cite{beldjilali2020, gezer2013, gherici2019, hretcanu2009, ozkan2014}). The~golden ratio and metallic means also have applications in the physical sciences, geometry, engineering, architecture, and~design. These applications provide additional motivation for the paper.}

%\medskip

{To obtain a nondegenerate geometric structure, we combine the rank-one tensor  \(\tilde g\) with a one-parameter family of Hessian metrics
\[
h_\lambda=\nabla_x^2\Phi_\lambda,
\qquad
\Phi_\lambda(x)=\sum_{i=1}^n J(x_i)+\lambda J(R),
\qquad \lambda\in\mathbb R.
\]

{\color{black}The Hessian metric \(h_\lambda\) determines the associated
\((1,1)\)-tensor \(A_\lambda\) by
\(
h_\lambda(A_\lambda X,Y)=\tilde g(X,Y),
\)
and its normalization gives the projector \(P_\lambda\).} The corresponding almost product, golden, and~metallic structures are obtained from \(P_\lambda\).

We study several  properties of these structures, including eigendistributions, parallelism, integrability, and~curvature. The~parameter \(\lambda\) deforms the Hessian metric by the term \(\lambda J(R)\),  producing a one-parameter family of Hessian metrics associated with the reciprocal cost~function.

The projector construction is general and can be applied to other rank-one tensors and nondegenerate metrics. In~this paper, we consider the projector induced by the rank-one tensor \(\tilde g\) and the Hessian metrics \(h_\lambda\) arising from reciprocal cost~geometry.

{\color{black}The paper is organized as follows. Section~\ref{sec:prelim} introduces the reciprocal cost geometry and the associated rank-one Hessian tensor \(\tilde g\). In~Section~\ref{sec3}, we construct the cost-induced projector \(P_\lambda\) and the induced almost product, golden, and~metallic structures. Section~\ref{sec:ndim} presents the general \(n\)-dimensional construction and studies the properties of the induced structures, including their eigendistributions, parallelism, integrability, and~curvature, with~the main results summarized in Theorem~\ref{thm:main-ndim}. Finally, Section~\ref{sec:2d} illustrates the construction in the two-dimensional case, where explicit formulas for \(A_\lambda\) and \(P_\lambda\) are derived.}

\section{Definitions and Basic~Properties}\label{sec:prelim}
 
%%%
%%%%%%
Let \((M,g)\) be a Riemannian manifold, let \(I\) denote the identity on \(TM\), and~let \(Q\colon TM\to TM\) be a \((1,1)\)-tensor~field.

\begin{Definition}
A $(1,1)$-tensor field $Q$ on $M$ is called a {polynomial structure} %MDPI: We added the italics. Please confirm this revision. % MZ: Use roman font here (not italics) to emphasize it, since the definition is already italicized.
 if it satisfies a polynomial relation of the form\vspace{-6pt}
\begin{align}\label{E-pol-n}
 Q^n + a_{n-1}Q^{n-1} + \cdots + a_1 Q + a_0 I = 0,
\end{align}
where $I$ is the identity operator on $TM$ and $a_i\in\mathbb{R}$.
\end{Definition}
Golden and metallic structures are special cases of polynomial structures.  In~particular,
\(Q^2=-I\) defines an almost complex structure,
\(Q^2=I\) defines an almost product structure, and~\(Q^2=0\) defines an almost tangent structure (see, e.g.,~\cite{Yano}).

\begin{Definition}
For integers $p,q$, a~$(1,1)$-tensor field $Q$ is called a {$(p,q)${-metallic structure}} %MDPI: We added the italics. Please confirm this revision. % MZ: Use roman font here (not italics) to emphasize it, since the definition is already italicized.
 if
\begin{equation}\label{E-metallic}
Q^2 = p\,Q + q\,I.
\end{equation}
\end{Definition}
%\medskip

A Riemannian metric $g$ is called {\it $Q${-compatible}} if
\begin{equation}\label{comp}
g(X,QY)=g(QX,Y), \qquad X,Y\in\Gamma(TM).
\end{equation}
When \(Q\) is a \((p,q)\)-metallic structure and \(g\) is \(Q\)-compatible, the~pair \((g,Q)\) is called a \emph{{metallic Riemannian structure}}. In~the particular case \(p=q=1\),
the pair \((g,Q)\) is called a \emph{{golden Riemannian structure}}
\cite{crasmareanu2008,hretcanu2007}. 
%\medskip

Replacing $X$ by $QX$ in \eqref{comp} and using \eqref{E-metallic}, we obtain
\[
g(QX,QY)=p\,g(X,QY)+q\,g(X,Y).
\]

% On every open subset where \%(h_\lambda\) is positive definite,
%\(
%(\mathbb %R_{>0}^n,h_\lambda,G_\lambda)
%\)
%defines a golden Riemannian %manifold in the sense of
%\cite{hretcanu2007,crasmareanu2008}.

It is known that a decomposition of the tangent bundle of a differentiable manifold $M$
into complementary distributions can be described in terms of projector operators.
For instance, let $T_1,\ldots,T_k$ be differentiable distributions on $M$ such that for every point $p\in M$, one has
\[
T_pM = T_1(p)\oplus \cdots \oplus T_k(p).
\]
This decomposition can be equivalently expressed by a family of $(1,1)$-tensor
fields $\pi_i$, $i=1,\ldots,k$, called {\it {projectors}}, satisfying
\[
\sum_{i=1}^k \pi_i = I, \qquad \pi_i\pi_j = \delta^i_j\,\pi_i,
\]
where $\delta^i_j$ are the Kronecker symbols. In~this case, $T_i=\operatorname{Im}(\pi_i)$.

In the case \(k=2\), such a decomposition determines an almost product structure.
Indeed, if~$\pi$ is one of the projectors, then define
\[
F = 2\pi - I,
\]
and obtain a $(1,1)$-tensor field satisfying $F^2 = I$. 

Conversely, any almost product structure \(F\)  induces the complementary projectors
\[
\pi^+=\frac12(I+F), \qquad \pi^-=\frac12(I-F),
\]
and the decomposition
\[
T_pM=T^+(p)\oplus T^-(p),
\]
where
\[
T^\pm(p)=\{v\in T_pM : Fv=\pm v\}.
\]

\begin{Theorem}[\cite{crasmareanu2008}]
Let $(M,g,Q)$ be a golden Riemannian manifold. Then,
\begin{equation}\label{2.4}
Q^{n} = f_n Q + f_{n-1} I
\end{equation}
for every integer $n>0$, where $(f_n)_n$ is the Fibonacci sequence.
\end{Theorem}

Using Binet's formula, relation \eqref{2.4} can be written as
\[
Q^{n}
=
f_n Q + f_{n-1} I
=
\frac{\varphi^{n} - (1-\varphi)^{n}}{\sqrt{5}}\, Q
+
\frac{\varphi^{n-1} - (1-\varphi)^{n-1}}{\sqrt{5}}\, I,
\]
for every positive integer \(n\).

\subsection*{{Reciprocal Cost~Geometry}} %MDPI: If there is only one subsection within a section, it should not be numbered. We have thus removed this section number. Please confirm.
\label{sec:cost-geometry} 

The main point of our construction is related to the canonical reciprocal cost function \(
J:\mathbb{R}_{>0}\to\mathbb{R},\)
\begin{equation}\label{cost}
J(x)=\frac12(x+x^{-1})-1,
\end{equation}
which is the unique solution of the polynomial composition law together with the curvature calibration (for more details, see~\cite{WZ}). The~function \(J\) is reciprocal,
\(
J(x)=J(x^{-1}),
\)
non-negative, with~a minimum at \(x=1\). {\color{black} The term \emph{{cost}} comes from the fact that \(J\) is non-negative, vanishes only at the balanced state \(x=1\), and~increases as the ratio \(x\) departs from \(1\) in either direction. Thus, \(J(x)\) measures the cost of imbalance. The~term \emph{{reciprocal}} refers to the symmetry
\(
J(x)=J(x^{-1}),
\)
which means that a ratio and its reciprocal have the same cost. This interpretation is used in the reciprocal cost geometry developed in~\cite{WZ,WZP}, where \(J\) plays a role similar to a divergence or discrepancy function in optimization and information geometry.} In logarithmic coordinates,
\[
J(e^t)=\cosh(t)-1.
\]
Near \(t=0\), one has
\[
J(e^t)=\frac{t^2}{2}+O(t^4).
\]
{\color{black}Thus, near~the balanced state \(x=1\), the~reciprocal cost function is approximated by the quadratic function \(\frac12 t^2\). The~function \(J\) extends this quadratic model to all positive ratios while preserving reciprocity and convexity.}

Among many possible multidimensional extensions, the~form considered here is motivated by the multiplicative structure of the one-dimensional reciprocal cost and by the logarithmic representation
\(
J(e^t)=\cosh(t)-1.
\)

We study the family of reciprocal cost functions (see, e.g.,~\cite{WZP})
\begin{equation}\label{mult}
J(x_1,\dots,x_n)=\frac12(R+R^{-1})-1,
\qquad R=\prod_{i=1}^n x_i^{\alpha_i},
\end{equation}
where \(x_i>0\) and \(\alpha=(\alpha_1,\ldots,\alpha_n)\in\R^n\setminus\{0\}\). In~logarithmic coordinates \(t_i=\log x_i\), the~function \eqref{mult} takes the form
\[
J(t)=\cosh(\alpha\cdot t)-1.
\]

Therefore, the~function \(J\) depends only on the scalar  
\(
S(t)=\alpha\cdot t=\sum_{i=1}^n \alpha_i t_i.
\)  Its Hessian is  the rank-one tensor
\[
\nabla^2J
=
\cosh(\alpha\cdot t)
\left(
\sum_{i=1}^n \alpha_i\,dt_i
\right)
\otimes
\left(
\sum_{i=1}^n \alpha_i\,dt_i
\right).
\] 
Since \(\nabla^2J\) has a rank of one, the~tensor
\(
\tilde g:=\nabla^2J
\)
is degenerate for \(n\ge2\).

The induced geometry is degenerate, with~a distinguished direction generated by \(\alpha\) and an integrable \((n-1)\)-dimensional null distribution. In~particular, the~ambient space is \(n\)-dimensional, and the associated Hessian structure in logarithmic coordinates reduces to a one-dimensional geometry. Hessian geometry and Hessian manifolds play an important role in affine differential geometry and information geometry (see, e.g.,~\cite{amari,shima2007}).

{The rank-one property of this Hessian tensor motivates the construction of additional geometric structures. To~obtain a nondegenerate geometric structure, we combine the rank-one tensor associated with the reciprocal cost geometry with a family of Hessian metrics. This construction produces an associated \((1,1)\)-tensor field whose normalization defines a projector. The~projector then induces an almost product structure and the corresponding golden and metallic~structures.

{\color{black}The next section develops the projector construction together with the induced almost product, golden, and~metallic structures. The~general \(n\)-dimensional construction is presented in Section~\ref{sec:ndim}, and~the two-dimensional case is studied in Section~\ref{sec:2d}.}

\section{The Cost-Induced Projector, Golden, and~Metallic~Structures}\label{sec3}

Let $M$ be a smooth manifold, let $g$ be a nondegenerate metric on $M$, and~let $\tilde g$ be a positive semidefinite symmetric $(0,2)$-tensor field of rank one. %Since $\tilde g$ has rank one, 
On an open set $U\subseteq M$ where $\tilde g\ne 0$, there exists a vector field $V$ such that
\begin{equation}\label{metric}
\tilde g(X,Y)=g(V,X)\,g(V,Y),
\qquad X,Y \in \Gamma(TU).
\end{equation}
The associated $(1,1)$-tensor field $A$ is defined by
\begin{equation}\label{aa}
g(AX,Y) = \tilde g(X,Y)
\qquad  X,Y \in \Gamma(TU).
\end{equation} 
%In local coordinates, the last equation becomes
%\[
%A^{i}{}_{j} =g^{ik}\tilde g_{kj},
%\]
%where $g_{ik}g^{kj}=\delta_i^j$. 
On $U$, using \eqref{metric}, we obtain
\begin{equation}\label{defA}
AX = g(V,X)\,V,
\qquad X\in \Gamma(TU).
\end{equation}
We have
\[g(AX,Y)=\tilde g(X,Y)=\tilde g(Y,X)=g(AY,X)=g(X,AY)\]

\begin{Lemma}\label{lem31}
The tensor $A$ defined by \eqref{aa} satisfies
\[
A^2 = \mu A,
\] where
$\mu = g(V,V).$
Moreover, 
\(
\mu = \tr(A).
\)
\end{Lemma}
\begin{proof}
By \eqref{defA}, 
\(
A^2X =  A\bigl(g(V,X)V\bigr)
= g(V,V)\,AX,
\)
so $A^2=\mu A$ with $\mu=g(V,V)$.
Consequently,%From $AX=g(V,X)V$, we obtain
\[
\tr(A)=g(V,V)=\mu.
\]
\end{proof}
\begin{Corollary}\label{cor1}
On the open subset \(U\subseteq M\) where
\(
\mu=g(V,V)\neq0,
\)
the tensor
\begin{equation}\label{proojj}
P := \frac{1}{\mu}A
\end{equation}
is a projector. Moreover,
\[
\im(P)=\mathrm{span}\{V\},
\qquad
\ker(P)=\{X\in TM|_U : g(V,X)=0\}.
\]
Hence,
\(
TM|_U=\im(P)\oplus\ker(P).
\)
\end{Corollary}

\begin{proof}
Since \(A^2=\mu A\) and \(\mu\neq0\), we have
\[
P^2=\frac1{\mu^2}A^2=\frac1\mu A=P,
\]
so \(P\) is a projector. Moreover,
\[
PX=\frac{g(V,X)}{g(V,V)}\,V,
\]
hence, \(\im(P)=\operatorname{span}\{V\}\) and\vspace{-4pt}
\[
\ker(P)=\{X\in TM|_U:\ g(V,X)=0\}.
\]
Finally, every vector field $X\in \Gamma(TM|_U)$ decomposes as
\(
X=PX+(X-PX),
\)
where $PX\in\im(P)$ and $X-PX\in\ker(P)$. Therefore,
\(
TM|_U=\im(P)\oplus\ker(P).
\)
\end{proof}

%%%OK je
We now use the projector \(P\) to construct the induced almost product, golden, and~metallic structures. Starting from the projector $P$ given by \eqref{proojj} and induced splitting\vspace{-4pt}
\[
TM|_U = \im(P)\oplus\ker(P),
\]  we obtain these structures.  The~constructions below follow from the identity $P^2=P$ and hold for any projector. The~reciprocal cost geometry provides a particular projector to which this construction is applied.
Let us define\vspace{-4pt}
\[
F := 2P - I.
\]

\begin{Proposition}
The tensor $F$ satisfies
\[
F^2 = I.
\]
Moreover, $F|_{\im(P)} = I$ and $F|_{\ker(P)} = -I$.
\end{Proposition}

\begin{proof}
Since $F=2P-I$ and $P$ is a projector, i.e.,~$P^2=P$, we have
\[
F^2 = (2P-I)^2 = 4P^2 - 4P + I = I.
\]
Moreover, for~$X\in\im(P)$, one has $PX=X$, so
\(
FX= X.
\)
For $X\in\ker(P)$, one has $PX=0$, so
\(
FX= -X.
\)
\end{proof}
%\medskip

Let us now consider an operator 
of the form
\[
G = \alpha P + \beta (I-P),
\qquad \alpha,\beta\in\mathbb{R}.
\]
Since $P^2=P$ and $P(I-P)=0$, we have
\[
G^2 = \alpha^2 P + \beta^2 (I-P).
\]
We require that $G$ satisfy the golden equation
\(
G^2 = G + I;
\)
then,
\[
\alpha^2 = \alpha + 1,
\qquad
\beta^2 = \beta + 1.
\]
Thus, $\alpha$ and $\beta$ are roots of the equation\vspace{-4pt}
\[
x^2 = x + 1,
\]
which has two solutions:\vspace{-4pt}
\[
\varphi=\frac{1+\sqrt5}{2}, 
\qquad 
1-\varphi=\frac{1-\sqrt5}{2}.
\]
{\color{black}Since the roots are distinct, we define the golden structure by taking the eigenvalue
\(\varphi\) on \(\operatorname{Im}(P)\) and \(1-\varphi=-\varphi^{-1}\) on
\(\ker(P)\):
\[
\alpha=\varphi,\qquad
\beta=1-\varphi=-\varphi^{-1}.
\]}
Therefore,
\[
G=\varphi P+(1-\varphi)(I-P).
\]
or equivalently, since $F=2P-I$,
\begin{equation}\label{ggold}
G=\frac12(I+\sqrt5\,F).
\end{equation}
A direct computation shows that \(G\) satisfies
\[
G^2=G+I.
\]
Thus, the golden structure is induced by the~projector.

\begin{Corollary}
The operator $G$ has eigenvalues $\varphi$ on $\im(P)$ and $-\varphi^{-1}$ on $\ker(P)$.
\end{Corollary}

\begin{proof}
On $\im(P)$, one has $F=I$; hence, $G=\frac{1}{2}(1+\sqrt{5})I=\varphi I$.
On $\ker(P)$, one has $F=-I$; hence, $G=\frac{1}{2}(1-\sqrt{5})I=-\varphi^{-1}I$.
\end{proof}

%\medskip

Let us now, for~ \(p,q\in\mathbb N\), define
\begin{equation}\label{metal}    
M_{p,q}
:= \frac{p}{2}I + \frac{1}{2}\sqrt{p^2+4q}\,F.
\end{equation}

\begin{Theorem}
The operator $M_{p,q}$ given by \eqref{metal} satisfies
\[
M_{p,q}^2 = p\,M_{p,q} + q\,I.
\]
\end{Theorem}

{  \begin{proof}
Using \(F^2=I\), from~\eqref{metal}, we obtain
\[
M_{p,q}^2
=
\left(
\frac{p}{2}I+\frac12\sqrt{p^2+4q}\,F
\right)^2
=
\left(
\frac{p^2}{4}+\frac{p^2+4q}{4}
\right)I
+
\frac{p}{2}\sqrt{p^2+4q}\,F.
\]
Therefore,
\[
M_{p,q}^2
=
\left(
\frac{p^2}{2}+q
\right)I
+
\frac{p}{2}\sqrt{p^2+4q}\,F.
\]
On the other hand,
\[
p\,M_{p,q}+qI
=
p\left(
\frac{p}{2}I+\frac12\sqrt{p^2+4q}\,F
\right)+qI=
\left(
\frac{p^2}{2}+q
\right)I
+
\frac{p}{2}\sqrt{p^2+4q}\,F.
\]
which coincides with
\(
M_{p,q}^2.\)
\end{proof}}
{\color{black}By construction, each operator in the family \(\{M_{p,q}\}_{p,q\in\mathbb N}\) is a
polynomial expression in the projector \(P\), and~the golden structure \(G=M_{1,1}\)
corresponds to the case \(p=q=1\). The~golden case \(p=q=1\) is distinguished within
this construction, while the general metallic structures \(M_{p,q}\) are obtained by
choosing the parameters \(p,q\) and  thus form  a generalization of the golden case.}
%\medskip

The following proposition gives the main properties of a general metallic structure.
\begin{Proposition}[Hreţcanu--Crăşmăreanu~\cite{hretcanu2013}]\label{prop2}
Let $M_{p,q}$ be defined by
\[
M_{p,q} = \frac{p}{2}I + \frac{1}{2}\sqrt{p^2+4q}\,F,
\qquad p,q\in\mathbb N,
\]
where $F^2=I$. Then, the following properties hold:

\begin{enumerate}
\item For every integer $n\ge 1$,
\[
M_{p,q}^n = G(n)\,M_{p,q} + q\,G(n-1)\,I,
\]
where $(G(n))_{n\ge 0}$ is the generalized secondary Fibonacci sequence defined by
\[
G(n+1)=pG(n)+qG(n-1),
\qquad
G(0)=0,\quad G(1)=1.
\]

\item The operator $M_{p,q}$ is an isomorphism on each tangent space $T_xM$, hence invertible. Its inverse is polynomial (of quadratic type, but~not metallic) and is given by
\[
\bar M_{p,q}=M_{p,q}^{-1} = \frac{1}{q}M_{p,q} - \frac{p}{q}I.
\]
It satisfies
\[
q \bar M_{p,q}^{2} + p \bar M_{p,q} - I = 0.
\]

\item The eigenvalues of $M_{p,q}$ are
\[
\frac{p+\sqrt{p^2+4q}}{2},
\qquad
\frac{p-\sqrt{p^2+4q}}{2}.
\]
\end{enumerate}
\end{Proposition} 
 
{\color{black} Proposition~\ref{prop2} plays an essential role in the present construction.
In our construction, the~eigenvalues of the induced metallic operator are the metallic
mean \(\sigma_{p,q}\) and its algebraic conjugate, realized on
\(\operatorname{Im}(P)\) and \(\ker(P)\), respectively.
Thus, the~eigenspace decomposition of the cost-induced metallic structure is determined
by the projector \(P\).}

\section{{The n-Dimensional Hessian~Construction}}\label{sec:ndim}
In this section, we construct the cost-induced projector in the \(n\)-dimensional case.
{\color{black} 
To obtain a nondegenerate geometric structure, we combine the rank-one tensor
\(\tilde g\) with a one-parameter family of Hessian metrics generated by the potential
\[
\Phi_\lambda(x_1,\ldots,x_n)
=
\sum_{i=1}^n J(x_i)
+
\lambda J(R),
\qquad
R=\prod_{i=1}^n x_i^{\alpha_i},
\]
where \(\alpha=(\alpha_1,\ldots,\alpha_n)\in\mathbb{R}^n\setminus\{0\}\) and
\(\lambda\in\mathbb{R}\). The~first term is the separable reciprocal cost,
while the second introduces an interaction through the same reciprocal cost
function. } The associated Hessian metric is
\begin{equation}\label{hh-lambda}
h_\lambda
=
\nabla_x^2\Phi_\lambda.
\end{equation}
The parameter \(\lambda\) is the deformation parameter.
For \(\lambda=0\), we obtain 
\[
h_0
=
\operatorname{diag}(x_1^{-3},\ldots,x_n^{-3}),
\]
which is positive definite on \(\mathbb R_{>0}^n\). % On domains where \(h_\lambda\) is positive definite, the~corresponding geometry is  related to Hessian and information geometry in the sense of~\cite{amari,shima2007}.

For \(\lambda\neq 0\), the~term \(\lambda J(R)\) introduces mixed second derivatives, so \(h_\lambda\) is generally nondiagonal. Its positive definite locus depends on \((x,\lambda)\).
Moreover, {\color{black}for every fixed point \(x=(x_1,\dots,x_n)\),} positive definiteness is preserved for sufficiently small values of \(|\lambda|\). 
An analysis of the signature and singular loci of \(h_\lambda\) in the $n$-dimensional case lies beyond the scope of the present paper. We work on open subsets where \(h_\lambda\) is positive~definite. 

  In the two-dimensional case, explicit conditions for positive definiteness are given in the following example.
\begin{Example}
{\color{black}Let \(n=2\), \(\alpha=(1,-1)\), so that \(R=x/y\). Then,
\[
h_\lambda=
\begin{pmatrix}
a & b\\
b & d
\end{pmatrix},
\qquad
a=\frac{1+\lambda y}{x^3},\quad
b=-\lambda\frac{x^2+y^2}{2x^2y^2},\quad
d=\frac{1+\lambda x}{y^3},
\]}
and
\[
\det(h_\lambda)
=ad-b^2
=\frac{4xy\bigl(1+\lambda(x+y)\bigr)-\lambda^2(x^2-y^2)^2}{4x^4y^4}.\]
{By Sylvester's criterion, \(h_\lambda\) is positive definite at \((x,y)\in\Rp^2\) if and only if \(a>0\) and \(\det(h_\lambda)>0\) or equivalently,}
\[
1+\lambda y>0
\quad\text{and}\quad
4xy\bigl(1+\lambda(x+y)\bigr)-\lambda^2(x^2-y^2)^2>0.
\]
\end{Example}
\begin{Remark}
For \(\lambda\neq0\), the~metric \(h_\lambda\) is not positive definite on  \(\mathbb R_{>0}^2\). For~fixed \(x>0\) and sufficiently large \(y\), 
\(
\det(h_\lambda)<0
\)
 because the term
\(
-\lambda^2 y^4
\)
dominates the numerator. Therefore, we restrict the construction to open subsets of \(\mathbb R_{>0}^2\) where \(h_\lambda\) is positive definite.\end{Remark}

\begin{Remark}
For \(\lambda=0\), the~metric \(h_0\) is positive definite on \(\mathbb R_{>0}^n\). Since the coefficients of \(h_\lambda\) depend on \(\lambda\) and on variables $x_i$,  this property is preserved for sufficiently small values of \(|\lambda|\) at every fixed point \((x_1,\dots,x_n)\).
We consider open subsets where \(h_\lambda\) is positive definite.\end{Remark}

%We now describe the construction of a projector in the multidimensional case. Let
%\[
%\Phi_0(x_1,\ldots,x_n)
%=
%\sum_{i=1}^n J(x_i).
%\]
%Then
%\begin{equation}\label{hh}
%h=\nabla^2\Phi_0
%=
%\operatorname{diag}%(x_1^{-3},\ldots,x_n^{-3}),
%\end{equation}
%which is positive definite and therefore nondegenerate on %\(\Rp^n\).
%\medskip

Recall that, in~logarithmic coordinates \(t_i=\log x_i\), the~reciprocal cost function \eqref{mult} has the form
\[
J(t)=\cosh(\alpha\cdot t)-1,
\qquad
\alpha\cdot t=\sum_{i=1}^n\alpha_i t_i.\]
with Hessian
\[
\nabla^2J
=
\cosh(\alpha\cdot t)\,\alpha\otimes\alpha.
\]

In the original \(x\)-coordinates, let
\[
\omega
=
\sum_{i=1}^n \alpha_i\frac{dx_i}{x_i}.
\]
Then, the  rank-one tensor is
\begin{equation}\label{ggg}
\tilde g
=
\cosh(\alpha\cdot t)\,\omega\otimes\omega.
\end{equation}
The associated \((1,1)\)-tensor field \(A_\lambda:=h_\lambda^{-1}\tilde g\) is defined by
\[
h_\lambda(A_\lambda X,Y)=\tilde g(X,Y),
\qquad
X,Y\in\Gamma(T\Rp^n).
\]
Let \(V_\lambda:=h_\lambda^{-1}(\omega)\) defined by \(h_\lambda(V_\lambda,X)=\omega(X)\) for all \(X\in\Gamma(T\Rp^n)\). Then, by~\eqref{ggg}, \mbox{we obtain}
\[\aligned
\tilde g(X,Y)
&=
\cosh(\alpha\cdot t)\,\omega(X)\omega(Y)\\
&=
\cosh(\alpha\cdot t)\,\omega(X)\,h_\lambda(V_\lambda,Y).
\endaligned
\]

Applying Lemma~\ref{lem31} and Corollary~\ref{cor1} with \(g=h_\lambda\), \(V=V_\lambda\), and~\(\tilde g\) as in \eqref{ggg}, we get}
\begin{equation}\label{aaaan}
A_\lambda X
=
\cosh(\alpha\cdot t)\,\omega(X)V_\lambda,
\qquad
A_\lambda=\cosh(\alpha\cdot t)\,V_\lambda\otimes\omega,
\end{equation}
 and 
\[
A_\lambda^2
=
\mu_\lambda A_\lambda,
\qquad
\mu_\lambda
=
\cosh(\alpha\cdot t)\,\omega(V_\lambda)
=
\tr(A_\lambda).
\]
We consider open subset of \(\mathbb R_{>0}^n\) where
\(
\mu_\lambda\neq0.
\)
On this set, the  tensor
\begin{equation}\label{proj}
P_\lambda=\frac1{\mu_\lambda}A_\lambda
\end{equation}
is a projector. More precisely,
\begin{equation}\label{projector}
P_\lambda X
=
\frac{\omega(X)}{\omega(V_\lambda)}\,V_\lambda.
\end{equation}

Since \(\tilde g\) is symmetric,
\[
h_\lambda(A_\lambda X,Y)
=
h_\lambda(X,A_\lambda Y),
\]
hence, \(A_\lambda\) and~therefore \(P_\lambda\) are self-adjoint with respect to \(h_\lambda\). 
 Applying the construction given in Section~\ref{sec3}, we obtain the induced almost product, golden, and~metallic structures
\begin{equation}\label{almsotp}
F_\lambda=2P_\lambda-I,
\qquad
F_\lambda^2=I,
\end{equation}
\begin{equation}\label{golden}
G_\lambda
=
\frac12(I+\sqrt5\,F_\lambda),
\qquad
G_\lambda^2=G_\lambda+I,
\end{equation}
{ and for~\(p,q\in\mathbb N\),}
\begin{equation}\label{metalic}
M_{p,q}^{\lambda}
=
\frac p2 I
+
\frac12\sqrt{p^2+4q}\,F_\lambda,
\qquad
(M_{p,q}^{\lambda})^2
=
pM_{p,q}^{\lambda}+qI.
\end{equation}

{ Consequently, \(F_\lambda\), \(G_\lambda\), and~\(M^\lambda_{p,q}\) are also self-adjoint with respect to \(h_\lambda\), since they are polynomial expressions in the self-adjoint operator \(P_\lambda\) with real coefficients.}

\subsection{Properties of the Induced~Projector}\label{sec:properties}

We now study properties of the projector \(P_\lambda\) constructed in the previous sections. Let \(\nabla^\lambda\) denote the Levi-Civita connection of the Hessian metric \(h_\lambda=\nabla_x^2\Phi_\lambda\) on the open subset of \(\Rp^n\) where \(h_\lambda\) is positive definite. We denote by $D$ the canonical flat affine connection associated with the chosen affine structure. In~affine coordinates, its Christoffel symbols vanish identically.
%\medskip

We first check whether \(P_\lambda\) is parallel with respect to the canonical flat affine connection in logarithmic coordinates, using the two-dimensional~case.

\begin{Example}
Let \(n=2\), \(\alpha=(1,-1)\), and~\(\lambda=0\). In~logarithmic coordinates
\(
u=\log x,
\)
\(v=\log y,
\)
we have
\[
\omega=du-dv.
\]
For the metric \(h_0\), we have
\[
V_0=h_0^{-1}(\omega)
=
e^u\partial_u-e^v\partial_v,
\]
and
\[
\omega(V_0)=e^u+e^v.
\]
Therefore,
\[
P_0
=
\frac1{e^u+e^v}
\begin{pmatrix}
e^u & -e^u\\
-e^v & e^v
\end{pmatrix}.
\]
Hence,
\[
\partial_u(P_0)^1{}_1
=
\frac{e^{u+v}}{(e^u+e^v)^2}\neq0.
\]
Since the canonical flat affine connection \(D\) satisfies
\[
D_{\partial_u}\partial_u
=
D_{\partial_u}\partial_v
=
D_{\partial_v}\partial_u
=
D_{\partial_v}\partial_v
=
0,
\]
we obtain
\[
(D_{\partial_u}P_0)^1{}_1
=
\partial_u(P_0)^1{}_1
\neq0.
\]
Therefore
\(
DP_0\neq0;
\)
i.e., the projector \(P_0\) is not parallel with respect to the canonical flat affine connection in logarithmic coordinates.
\end{Example}

\begin{Remark} The induced tensor \(G_\lambda\) satisfies
\[
G_\lambda^2=G_\lambda+I,
\]
and is self-adjoint with respect to the Hessian metric \(h_\lambda\). On~every open subset of \(\mathbb R_{>0}^n\) where \(h_\lambda\) is positive definite,
\(
(\mathbb R_{>0}^n,h_\lambda,G_\lambda)
\)
defines a golden Riemannian manifold.
\end{Remark}

The following theorem describes the family of linear connections preserving the induced golden structure \(G_\lambda\).

\begin{Theorem}[Theorem 5.1,~\cite{crasmareanu2008}]
Let \(F_\lambda=2P_\lambda-I\) be the induced almost product structure associated with the golden structure \(G_\lambda\). Then, the set of linear connections \(\nabla\) satisfying
\(
\nabla G_\lambda=0
\)
is given by
\[
\nabla_XY=
\frac15
\Bigl[
3\widetilde\nabla_XY
+
2G_\lambda(\widetilde\nabla_XG_\lambda Y)
-
G_\lambda(\widetilde\nabla_XY)
-
\widetilde\nabla_XG_\lambda Y
\Bigr]
+
\mathcal O_{F_\lambda}Q(X,Y),
\]
where \(\widetilde\nabla\) is an arbitrary fixed linear connection and \(Q\) is a \((1,2)\)-tensor field for which $\mathcal O_{F_\lambda}Q$ is an associated Obata operator
\[
\mathcal O_{F_\lambda}Q(X,Y)
=
\frac12
\Bigl[
Q(X,Y)
+
F_\lambda Q(X,F_\lambda Y)
\Bigr].
\]
\end{Theorem}

%\begin{Example}
%Let \(n=2\), \(\alpha=(1,-1)\), and \%(\lambda=1\). In logarithmic coordinates
%\(
%u=\log x,\qquad v=\log y,
%\)
%we have
%\[
%\omega=du-dv.
%\]
%For \(\lambda=1\), the projector has the component
%\[
%(P_\lambda)^1{}_1
%=
%\frac{\frac12 e^{2u}+e^u-\frac12 e^{2v}}{e^u+e^v}.
%\]
%Therefore
%\[
%\partial_u (P_\lambda)^1{}_1
%=
%\frac{
%e^u\left((e^u+1)(e^u+e^v)-\frac12 %e^{2u}-e^u+\frac12 e^{2v}\right)
%}
%{(e^u+e^v)^2}.
%\]
%At the point \((x,y)=(2,1)\), that is \%(u=\log2\), \(v=0\), we obtain
%\[
%\left(D_{\partial_u}P_\lambda\right)^1{%}_1
%=
%\partial_u (P_\lambda)^1{}_1
%=
%\frac{11}{9}\neq0.
%\]
%Hence \(DP_\lambda\neq0\), so \%(P_\lambda\) is not parallel with respect to the canonical flat affine connection.
%\end{Example}

We now study the parallelism properties of \(P_\lambda\) with respect to the Levi-Civita connection \(\nabla^\lambda\) of \(h_\lambda\).

\begin{Example}
In the two-dimensional case, the~projector \(P_0\) is not parallel with respect to the Levi-Civita connection \(\nabla^0\) of \(h_0\). We have
\[
h_0=
\begin{pmatrix}
x^{-3} & 0\\
0 & y^{-3}
\end{pmatrix},
\]
and
\[
P_0=
\frac1{x+y}
\begin{pmatrix}
x & -x^2/y\\
-y^2/x & y
\end{pmatrix}.
\]
A direct computation gives
%\[
%(\nabla^{0}_{\partial_x}P_0)^x{}_x
%=
%\partial_x(P_0)^x{}_x
%+
%\Gamma^x_{xx}(P_0)^x{}_x
%-
%\Gamma^x_{xx}(P_0)^x{}_x
%=
%\partial_x(P_0)^x{}_x.
%\]
%Hence
\[
(\nabla^{0}_{\partial_x}P_0)^x{}_x
=
\partial_x\left(\frac{x}{x+y}\right)
=
\frac{y}{(x+y)^2}\neq0.
\]
Therefore, 
\(
\nabla^{0}P_0\neq0.
\)
\end{Example}

\begin{Example}
For \(\lambda\neq0\), the~projector \(P_\lambda\) is not parallel with respect to \(\nabla^\lambda\).

The coefficients of \(P_\lambda\) and the Christoffel symbols of \(\nabla^\lambda\) contain nontrivial mixed terms  from \(J(x/y)\).  For~example, at~the point \((x,y)=(1,1)\), we obtain
\[
\left(\nabla^\lambda_{\partial_x}P_\lambda\right)^x{}_x
=
\frac{1+\lambda}{4(1+2\lambda)}.
\]
At the point \((1,1)\), the~metric \(h_\lambda\) is positive definite only for \(1+2\lambda>0\), and~consequently $\left(\nabla^\lambda_{\partial_x}P_\lambda\right)^x{}_x\ne 0$. Hence, \(P_\lambda\) is not parallel with respect to the Levi-Civita connection \(\nabla^\lambda\).  
\end{Example}
The two-dimensional case represents the simplest nontrivial realization of the general construction. The~above examples show that the induced almost product, golden, and~metallic structures are  non-parallel in~general.

\subsection{Integrability of the~Eigendistributions}\label{sec:integrability}

{ Motivated by Proposition 5.3 of~\cite{crasmareanu2008}, we study the integrability of the distributions induced by the projector \(P_\lambda\). 

Since the almost product, golden, and~metallic structures
\(
F_\lambda=2P_\lambda-I,\)
\(G_\lambda,\)
\(M^\lambda_{p,q}
\)
are obtained from \(P_\lambda\), they have the same eigendistributions,
\(
\operatorname{Im}(P_\lambda)\) and \(
\ker(P_\lambda).
\)

The distribution
\(
\operatorname{Im}(P_\lambda)=\operatorname{span}\{V_\lambda\}
\)
is one-dimensional and therefore integrable. 
For \(\ker(P_\lambda)\), recall from \eqref{projector} that
\[
\ker(P_\lambda)=\ker(\omega),
\]
where
\[
\omega=\sum_{i=1}^{n}\alpha_i\,\frac{dx_i}{x_i}
      =d(\log R).
\]
In logarithmic coordinates, \(\omega=\sum_{i=1}^{n}\alpha_i\,dt_i\) has constant
coefficients, so \(\ker(\omega)\) is a constant \((n-1)\)-dimensional distribution
and is therefore~integrable.

Consequently, both eigendistributions of \(P_\lambda\) are integrable and~so are
the induced almost product, golden, and~metallic structures.}

%\medskip

\subsection{A Rank-One Representation of the Cost-Induced~Projector} Let \(P_\lambda\) be the cost-induced projector. Since
\[
P_\lambda X
=
\frac{h_\lambda(V_\lambda,X)}
{h_\lambda(V_\lambda,V_\lambda)}
\,V_\lambda,
\]
on the open set where
\(
h_\lambda(V_\lambda,V_\lambda)\neq0,
\)
we define
\[
\xi_\lambda:=V_\lambda,
\qquad
\eta_\lambda(X)
:=
\frac{
h_\lambda(V_\lambda,X)
}{
h_\lambda(V_\lambda,V_\lambda)
}.
\]
Then,
\(
\eta_\lambda(\xi_\lambda)=1
\)
and
\(
P_\lambda X
=
\eta_\lambda(X)\xi_\lambda.
\)
Therefore,
\[
P_\lambda=\eta_\lambda\otimes\xi_\lambda.
\]
Consequently, the~cost-induced golden structure \(G_\lambda\) given by \eqref{golden} and the metallic structures \(M_{p,q}^\lambda\) given by \eqref{metalic}
can be written in the form
\[
G_\lambda=(1-\varphi)I+\sqrt5\,\eta_\lambda\otimes\xi_\lambda,
\]
and
\[
M_{p,q}^\lambda
=
\left(
\frac p2-\frac12\sqrt{p^2+4q}
\right)I
+
\sqrt{p^2+4q}\,\eta_\lambda\otimes\xi_\lambda.
\]

Thus the reciprocal cost geometry together with the Hessian metric \(h_\lambda\) determines the projector \(P_\lambda\) and the associated polynomial structures. 
The vector field \(V_\lambda\) generates \(\im(P_\lambda)\), while \(\ker(P_\lambda)\) gives the complementary~distribution.

{ 
\subsection{Curvature Properties of the Hessian Metric \(h_\lambda\)}

We now consider  curvature properties of the Hessian metric
\(
h_\lambda=Dd\Phi_\lambda,
\)
where \(D\) denotes the canonical flat affine connection on
\(\mathbb R_{>0}^n\).

Let \(
\nabla^\lambda
\)
be the Levi-Civita connection of \(h_\lambda\). The~difference tensor \(K^\lambda\) of the Levi-Civita connection \(\nabla^\lambda\) of \(h_\lambda\) and  \(D\) is defined by
\begin{equation}
K^\lambda_XY
:=
\nabla^\lambda_XY-D_XY.
\end{equation} Since both \(\nabla^\lambda\) and \(D\) are symmetric connections, the~tensor \(K^\lambda\) is symmetric; i.e.,
\(
K^\lambda_XY
=
K^\lambda_YX.
\) Starting from 
\[
X(h_\lambda(Y,Z))
=
h_\lambda(\nabla^\lambda_XY,Z)
+
h_\lambda(Y,\nabla^\lambda_XZ).
\]
and substituting
\[
\nabla^\lambda_XY
=
D_XY+K^\lambda_XY,
\]
we have
\[
\aligned
X(h_\lambda(Y,Z))
&=
h_\lambda(D_XY,Z)
+
h_\lambda(K^\lambda_XY,Z)
\\
&\quad
+
h_\lambda(Y,D_XZ)
+
h_\lambda(Y,K^\lambda_XZ).
\endaligned
\]
Hence,
\[
(D_Xh_\lambda)(Y,Z)
=
h_\lambda(K^\lambda_XY,Z)
+
h_\lambda(Y,K^\lambda_XZ).
\]
Since \(h_\lambda\) is Hessian metric, the~cubic tensor
\begin{equation}\label{l1}
C_\lambda(X,Y,Z)
:=
(D_Xh_\lambda)(Y,Z)
\end{equation}
is totally symmetric. Therefore, 
\(
h_\lambda(K^\lambda_XY,Z)
=
h_\lambda(Y,K^\lambda_XZ),
\)
and consequently,
\begin{equation}\label{l2}
h_\lambda(K^\lambda_XY,Z)
=
\frac12(D_Xh_\lambda)(Y,Z).
\end{equation}
Therefore, \(K^\lambda\) is completely determined by the  tensor
\(
C_\lambda(X,Y,Z).
\)

Since the connection \(D\) is flat, the~curvature tensor of
\(\nabla^\lambda\) is given by
\[
R^\lambda(X,Y)Z
=
K^\lambda_XK^\lambda_YZ
-
K^\lambda_YK^\lambda_XZ
+
(D_XK^\lambda)(Y,Z)
-
(D_YK^\lambda)(X,Z).
\]
Using \eqref{l1}, \eqref{l2}, and~the Codazzi-type identity \((D_XK^\lambda)(Y,Z)=(D_YK^\lambda)(X,Z)\), we obtain
\[
R^\lambda(X,Y,Z,W)
=
\frac14
\Big(
h_\lambda^{-1}
\big(
C_\lambda(X,Z,\cdot),
C_\lambda(Y,W,\cdot)
\big)
-
h_\lambda^{-1}
\big(
C_\lambda(Y,Z,\cdot),
C_\lambda(X,W,\cdot)
\big)
\Big),
\]
where
\(
R^\lambda(X,Y,Z,W):
=
h_\lambda(R^\lambda(X,Y)Z,W)
\). 
Using \(C_\lambda(X,Y,Z)=2h_\lambda(K^\lambda_XY,Z)\), the~last formula becomes
\begin{equation}\label{l3}
R^\lambda(X,Y,Z,W)
=
\frac14
\Big(
h_\lambda(K^\lambda_XZ,K^\lambda_YW)
-
h_\lambda(K^\lambda_YZ,K^\lambda_XW)
\Big).
\end{equation}
Thus, the Riemannian curvature of \(h_\lambda\) is completely determined by the cubic tensor \(C_\lambda\).
If \(\{E_i\}_{i=1}^n\) is a local \(h_\lambda\)-orthonormal frame, then the coresponding Ricci tensor is 
\[
\operatorname{Ric}_{h_\lambda}(Y,Z)
=
\sum_{i=1}^n
R^\lambda(E_i,Y,Z,E_i).
\]
Using \eqref{l3}, we obtain
\[
\operatorname{Ric}_{h_\lambda}(Y,Z)
=
\frac14
\sum_{i=1}^n
\Big(
h_\lambda(K^\lambda_YZ,K^\lambda_{E_i}E_i)
-
h_\lambda(K^\lambda_{E_i}Z,K^\lambda_YE_i)
\Big).
\]
Therefore, the Ricci tensor is also completely determined by the cubic tensor \(C_\lambda\). 
The scalar curvature of \(h_\lambda\) is defined by
\(
\operatorname{Scal}_{h_\lambda}
=
\operatorname{tr}_{h_\lambda}
(\operatorname{Ric}_{h_\lambda}).
\)

%\medskip

In the two-dimensional case, the~Ricci tensor is completely determined by the scalar curvature. For~\(\lambda=0\), the~Hessian metric is
\(
h_0
=
\operatorname{diag}(x^{-3},y^{-3})
\), and~the corresponding curvature vanishes, and~\(
\operatorname{Scal}_{h_0}=0.
\)

For \(\lambda\neq0\),  the~scalar curvature is nonzero in general and depends  on  \(\lambda\). By~direct calculation, we have\vspace{-4pt}
\[
\operatorname{Scal}_{h_\lambda}
=
-\frac{
4\lambda x^2y^2
\Big(
\lambda(x+y)^3+3(x^2+y^2)
\Big)
}{
\Big(
\lambda^2(x^2-y^2)^2
-
4\lambda xy(x+y)
-
4xy
\Big)^2
}.
\]
Since the expression is not identically zero for \(\lambda\neq0\), the~Hessian metric \(h_\lambda\) is non-flat in general. Thus, the interaction term
\(
\lambda J\!\left({x}/{y}\right)
\)
produces nontrivial curvature of \(h_\lambda\).}

The preceding results can be summarized in the following~theorem.

\begin{Theorem}\label{thm:main-ndim}
Let \(\alpha\in\R^n\setminus\{0\}\), \(\lambda\in\R\), and~let \(U\subseteq\Rp^n\) be an
open set on which the Hessian metric \(h_\lambda=\nabla_x^2\Phi_\lambda\) is positive
definite and \(\mu_\lambda=\tr(A_\lambda)\neq0\). Then, the rank-one tensor
\(\tilde g=\cosh(\alpha\cdot t)\,\omega\otimes\omega\), with~\(\omega=d(\log R)\),
together with the Hessian metric \(h_\lambda\) induces the following structures on \(U\):
\begin{enumerate}
\item A cost-induced projector \(P_\lambda=\eta_\lambda\otimes\xi_\lambda\) with
\(\im(P_\lambda)=\operatorname{span}\{V_\lambda\}\) and \(\ker(P_\lambda)=\ker\omega\),
self-adjoint with respect to \(h_\lambda\);
\item The almost product structure \(F_\lambda=2P_\lambda-I\), with~eigenvalues \(+1\)
on \(\im(P_\lambda)\) and \(-1\) on \(\ker(P_\lambda)\), and~the golden and metallic
structures \(G_\lambda\) and \(M^\lambda_{p,q}\), all polynomial in \(P_\lambda\) and
hence \(h_\lambda\)-self-adjoint; on \(\im(P_\lambda)\) the operators \(G_\lambda\) and
\(M^\lambda_{p,q}\) act by the golden mean \(\varphi\) and by the metallic mean
\(\tfrac12(p+\sqrt{p^2+4q})\), respectively, and~by their algebraic conjugates on
\(\ker(P_\lambda)\);
\item Integrable eigendistributions, so that \(F_\lambda\), \(G_\lambda\), and~\(M^\lambda_{p,q}\) are integrable.
\end{enumerate}

{Moreover,} %MDPI: We added indention for this paragraph. Please confirm this revision. %MZ it is okay
 these structures are generally not parallel with respect to either the
canonical flat affine connection \(D\) or the Levi-Civita connection
\(\nabla^\lambda\) of the Hessian metric \(h_\lambda\). The~metric
\(h_\lambda\) is flat for \(\lambda=0\), whereas for \(\lambda\neq0\), it is
generally non-flat, with~curvature determined by the cubic tensor
\(C_\lambda=Dh_\lambda\).
\end{Theorem}

 \section{The Two-Dimensional~Case}\label{sec:2d}

We now consider the two-dimensional case as an illustration of the general construction. Consider the canonical reciprocal cost function \eqref{cost}. Following the $n$-dimensional case \eqref{mult}, define on \(\Rp^2\):\vspace{-4pt}
\begin{equation}
R(x,y):=\frac{x}{y},
\qquad
J(x,y):=\frac12\left(R+R^{-1}\right)-1.
\end{equation}
This corresponds to the choice \(\alpha=(1,-1)\) in the $n$-dimensional~model. 

\subsection{Logarithmic~Coordinates}

Introduce logarithmic coordinates
\(
u=\log x,
\,
v=\log y.
\)
Then,
\begin{equation}\label{2dcase}
R=e^{u-v},
\qquad
J(u,v)=\cosh(u-v)-1.
\end{equation}
In this case, the~function \(J(u,v)\) depends only on the quantity
\(
u-v.
\) 
The Hessian tensor 
\(
\tilde g=\nabla^2J
\) of \(J(u,v)=\cosh(u-v)-1\) is
%that is,
%\(
%\tilde g_{ij}
%\frac{\partial^2 J}{\partial u^i\partial u^j}.
%\)
%From \eqref{2dcase}, 
%we obtain
%\[
%\frac{\partial^2 J}{\partial u^2}
%=
%\cosh(u-v),
%\qquad
%\frac{\partial^2 J}{\partial v^2}
%=
%%\cosh(u-v),\qquad
%\frac{\partial^2 J}{\partial u\partial v}
%=
%-\cosh(u-v).
%\]
%Therefore
\begin{equation}\label{rankone}
\tilde g
=
\cosh(u-v)\,
(du-dv)\otimes(du-dv)=
\cosh(u-v)
\begin{pmatrix}
1 & -1\\
-1 & 1
\end{pmatrix}.
\end{equation}

Since \(\cosh(u-v)>0\), the~tensor \(\tilde g\) is positive semidefinite of rank one, with~distinguished direction \(V=\partial_u-\partial_v\).

\subsection{\((x,y)\)-Coordinates}

Passing to the original \((x,y)\)-coordinates, and~using
\(
u=\log x,\)
\(v=\log y,
\)
we obtain
\[
V=x\partial_x-y\partial_y.
\]
Using \eqref{rankone}, 
\(
du=\frac{dx}{x},
dv=\frac{dy}{y},
\)
and
\[
\cosh\!\left(\log\!\left(\frac{x}{y}\right)\right)
=
\frac12\left(\frac{x}{y}+\frac{y}{x}\right)
=
\frac{x^2+y^2}{2xy}
\]
we obtain
\[
\tilde g
=
\frac{x^2+y^2}{2xy}
\left(
\frac{dx}{x}-\frac{dy}{y}
\right)
\otimes
\left(
\frac{dx}{x}-\frac{dy}{y}
\right).
\]
Hence,
\begin{equation}\label{gtilda}
\tilde g
=
\frac{x^2+y^2}{2xy}
\begin{pmatrix}
x^{-2} & -\frac1{xy}\\[4pt]
-\frac1{xy} & y^{-2}
\end{pmatrix}.
\end{equation}
So, the~tensor \(\tilde g\) is a rank-one~tensor.

On the other hand, we can also consider the full Hessian of the cost function directly in the original \((x,y)\)-coordinates; i.e.,
\[
\tilde g_x=\nabla_x^2J(x,y).
\]
A direct computation gives
\[
\det(\tilde g_x)
=
-\frac{(x^2-y^2)^2}{4x^4y^4}.
\]
Thus, \(\tilde g_x\) is generically nondegenerate and indefinite,
while it becomes singular on the locus \(x=y\).
The geometry determined by the Hessian metric \(\tilde g_x\) was studied in~\cite{WZP}.

\begin{Remark}
The rank-one tensor \(\tilde g\) obtained from the logarithmic representation of the cost function and the full Hessian metric \(\tilde g_x\) in the original \((x,y)\)-coordinates carry different geometry. In~particular, \(\tilde g\) is degenerate of constant rank one and determines the distinguished comparison direction generated by
\(
V=x\partial_x-y\partial_y,
\)
while \(\tilde g_x\) is generically nondegenerate and becomes singular on the locus \(x=y\).

By Lemma~\ref{lem31} and Corollary~\ref{cor1}, the~rank-one property of \(\tilde g\), when paired with a nondegenerate metric, leads to the construction of a projector. The~tensor \(\tilde g_x\) does not have this property and does not produce the projector-type structures. For~this reason, the~projector construction developed in the following part
is associated with \(\tilde g\).
\end{Remark}

%%%%dovde

\subsection{The Induced~Projector}

{\color{black}The rank-one tensor \(\tilde g\) is degenerate and does not define a metric. To~obtain an operator from \(\tilde g\), we pair it with a nondegenerate reference metric, and~the reciprocal cost geometry supplies a natural one-parameter family of such metrics. We deform the separable cost \(J(x)+J(y)\) by the interaction term \(\lambda J(x/y)\), built from the same reciprocal cost function.}
 
We introduce a one-parameter family of nondegenerate Hessian metrics on \(\Rp^2\) by \begin{equation}\label{philambda}
\Phi_\lambda(x,y)
=
J(x)+J(y)+\lambda J\!\left(\frac{x}{y}\right),
\qquad \lambda\in\mathbb{R}.
\end{equation}
{For \(\lambda=0\), this reduces to \(\Phi_0(x,y)=J(x)+J(y)\), whose Hessian}
\begin{equation}\label{hhh}
h_0=\nabla^2\Phi_0
=
\begin{pmatrix}
x^{-3} & 0\\
0 & y^{-3}
\end{pmatrix}
\end{equation}
is positive definite on \(\Rp^2\) and it is used as a separable reference~metric.

 The associated \((1,1)\)-tensor field is defined by
\begin{equation}
A_0=h_0^{-1}\tilde g.
\end{equation}
The metric $h_0$ provides a nondegenerate reference metric for the distinguished comparison direction generated by \(V=x\partial_x-y\partial_y.
\)

The term \(\lambda J(x/y)\) introduces mixed second derivatives along $\omega$, where
\[
\omega=d\log(x/y)=\frac{dx}{x}-\frac{dy}{y},
\]
so the Hessian metric \(h_\lambda=\nabla_x^2\Phi_\lambda\) is nondiagonal for \(\lambda\neq0\). The~associated \((1,1)\)-tensor is
\[
A_\lambda=h_\lambda^{-1}\tilde g,
\]
and the projector is obtained by normalization, as~given in Section~\ref{sec3}.

\begin{Remark}
The construction involves two natural affine structures associated with the reciprocal cost~geometry.

In the two-dimensional case, the rank-one tensor \(\tilde g=\cosh(u-v)\,(du-dv)\otimes(du-dv)\) is derived with respect to the flat structure in the logarithmic coordinates \((u,v)\), while the Hessian metric \(h_\lambda=\nabla_x^2\Phi_\lambda\) is used with respect to the flat structure in the original \((x,y)\)-coordinates. The~$n$-dimensional construction considered in Section~\ref{sec:ndim} is based on the same choice of \mbox{affine structures.}
\end{Remark}

A direct computation of \(h_\lambda=\nabla_x^2\Phi_\lambda\) gives
\begin{equation}\label{hlamba}
h_\lambda=
\begin{pmatrix}
a & b\\
b & d
\end{pmatrix},
\qquad
a=\frac{1+\lambda y}{x^3},\quad
b=-\lambda\frac{x^2+y^2}{2x^2y^2},\quad
d=\frac{1+\lambda x}{y^3}.
\end{equation}
\begin{Proposition}
Let \(h_\lambda\) and \(\tilde g\) be given by \eqref{hlamba} and \eqref{gtilda}, respectively. If~\(
\det(h_\lambda)=ad-b^2\neq0,
\)
then
\begin{equation}\label{alambda}
A_\lambda
=
\frac{x^2+y^2}{2xy}\,
\frac1{ad-b^2}
\begin{pmatrix}
\dfrac{d}{x^2}+\dfrac{b}{xy}
&
-\dfrac{d}{xy}-\dfrac{b}{y^2}
\\[8pt]
-\dfrac{b}{x^2}-\dfrac{a}{xy}
&
\dfrac{b}{xy}+\dfrac{a}{y^2}
\end{pmatrix}.
\end{equation}
\end{Proposition}

\begin{proof}
By \eqref{hlamba},
\[
h_\lambda^{-1}
=
\frac1{ad-b^2}
\begin{pmatrix}
d & -b\\
-b & a
\end{pmatrix}.
\]
Multiplying \(h_\lambda^{-1}\) by the matrix of \(\tilde g\) given by \eqref{gtilda}, we obtain \eqref{alambda}.
\end{proof}

%\begin{Proposition}
%Let \(h\) be defined by \eqref{hhh} and $\tilde g$ by \eqref{gtilda}. Then
%\begin{equation}\label{aaaa}
%A
%=h^{-1}\tilde g=
%\frac{x^2+y^2}{2xy}
%\begin{pmatrix}
%x & -x^2/y\\
%-y^2/x & y
%\end{pmatrix}
%\end{equation}
%with respect to the coordinate frame
%\((\partial_x,\partial_y)\).
%\end{Proposition}

%\begin{proof}
%From \eqref{hhh}, we have
%\[
%h^{-1}
%=
%\begin{pmatrix}
%x^3 & 0\\
%0 & y^3
%\end{pmatrix}.
%\]
%Multiplying \(h^{-1}\) with the matrix of \(\tilde g\) given by \eqref{gtilda}, we obtain
%\eqref{aaaa}.
%\end{proof}

{\color{black}By Lemma~\ref{lem31} and Corollary~\ref{cor1}, applied with \(g=h_\lambda\), the~rank-one tensor \(\tilde g\) and the nondegenerate metric \(h_\lambda\) give \(A_\lambda=h_\lambda^{-1}\tilde g\) with
\[
A_\lambda^2=\mu_\lambda A_\lambda,
\qquad
\mu_\lambda=\tr(A_\lambda),
\]
and the normalized tensor \(P_\lambda:=\mu_\lambda^{-1}A_\lambda\) is a projector on the open subset of \(\Rp^2\) where \(ad-b^2\neq0\) and \(\mu_\lambda\neq0\).}

%\begin{Theorem}
%The tensor \(A\) given by %\eqref{aaaa} satisfies
%\(
%A^2=\mu A,
%\)
%where
%\[
%\mu
%=
%\frac{(x+y)(x^2+y^2)}{2xy}.
%\]
%Hence
%\[
%P:=\frac1\mu A
%\]
%is a projector on \(T\Rp^2\).
%\end{Theorem}

%\begin{proof}
%A direct computation gives
%\[
%A^2
%=
%frac{(x+y)(x^2+y^2)}{2xy}\,A.
%\]
%Since \(x,y>0\), we have \%(\mu>0\). Therefore
%\(
%P=\frac1\mu A
%\)
%is well defined and satisfies
%\(
%P^2=P.
%\)
%\end{proof}

\begin{Corollary}
The corresponding projector is
\[
P_\lambda(X)
=
\frac{\omega(X)}{\omega(V_\lambda)}\,V_\lambda,
\qquad
V_\lambda=h_\lambda^{-1}\omega,
\]
where
\[
\omega=\frac{dx}{x}-\frac{dy}{y}.
\]
Moreover,
\[
\im(P_\lambda)=\mathrm{span}\{V_\lambda\},
\qquad
\ker(P_\lambda)=\ker\omega .
\]
In particular, \(\ker(P_\lambda)\) is generated by
\(
x\partial_x+y\partial_y .
\)
\end{Corollary}

\begin{proof}
Since
\[
\tilde g
=
\frac{x^2+y^2}{2xy}\,\omega\otimes\omega
\]
and \(V_\lambda=h_\lambda^{-1}\omega\), the~tensor \(A_\lambda\) has the form
\[
A_\lambda X
=
\frac{x^2+y^2}{2xy}\,\omega(X)V_\lambda .
\]
After normalization, we have
\[
P_\lambda(X)
=
\frac{\omega(X)}{\omega(V_\lambda)}V_\lambda .
\]
Hence,
\(
\im(P_\lambda)=\mathrm{span}\{V_\lambda\}.
\) Also,
\[
P_\lambda X=0
\quad\Longleftrightarrow\quad
\omega(X)=0,
\]
so
\(
\ker(P_\lambda)=\ker\omega .
\)
Since
\[
\omega(x\partial_x+y\partial_y)
=
\frac{dx}{x}(x\partial_x+y\partial_y)
-
\frac{dy}{y}(x\partial_x+y\partial_y)
=
1-1=0,
\]
and \(\ker\omega\) is one-dimensional, we get
\[
\ker(P_\lambda)=\mathrm{span}\{x\partial_x+y\partial_y\}.
\]
%Since \(P\) is a rank-one projector on a two-dimensional manifold, both
%\(\operatorname{im}(P)\) and \(\ker(P)\) are one-dimensional. Hence the above non-zero vector fields determine these distributions.
\end{proof}

%\begin{Corollary}
%The canonical projector is
%\[
%P
%=
%\frac1{x+y}
%\begin{pmatrix}
%x & -x^2/y\\
%-y^2/x & y
%\end{pmatrix}.
%\]
%Moreover,
%\(
%\im(P)
%=
%\mathrm{span}\{x^2\partial_x-y^2\partial_y\},
%\)
%and
%\(
%\ker(P)
%=
%\mathrm{span}\%%{x\partial_x+y\partial_y\}.
%\)
%\end{Corollary}

%\begin{proof}
%The formula for \(P\) follows directly from \eqref{aaaa}. Furthermore,
%\[
%P(x^2\partial_x-y^2\partial_y)
%=
%x^2\partial_x-y^2\partial_y,
%\]
%hence
%\[
%x^2\partial_x-y^2\partial_y\in\im(P).
%\]
%Also,
%\(
%P(x\partial_x+y\partial_y)=0,
%\)
%so
%\[
%x\partial_x+y\partial_y\in\ker%(P).
%\]

%Since both distributions are one-dimensional, this determines
%\(\im(P)\) and \(\ker(P)\).
%\end{proof}

 \section{Conclusions}

In this paper, we studied a family of projector-induced polynomial structures associated with the reciprocal cost geometry.  Starting from the rank-one tensor \(\tilde g\) determined by the reciprocal cost function in logarithmic coordinates and a family of Hessian metrics \(h_\lambda\), we constructed  \((1,1)\)-tensor field \(A_\lambda\) and the projector \(P_\lambda\). The~projector \(P_\lambda\) determines \mbox{a splitting}
\[
TU=\im(P_\lambda)\oplus\ker(P_\lambda),
\]
which induces the almost product structure
\(
F_\lambda=2P_\lambda-I,
\)
and the corresponding golden and metallic structures. The~Hessian structure  is determined by the metric
\(
h_\lambda=\nabla_x^2\Phi_\lambda
\)
with respect to  the \(x\)-coordinates, while the rank-one tensor \(\tilde g\)  is obtained from the logarithmic reciprocal cost geometry. {\color{black}The rank-one degeneracy is therefore not an obstruction but~the starting point of the construction.}

We studied several properties of the projector \(P_\lambda\) and the induced structures, including eigendistributions,  parallelism, integrability, and~curvature. In~particular, we showed that \(P_\lambda\) is generally not parallel with respect to either the canonical flat affine connection in logarithmic coordinates or the Levi-Civita connection of the Hessian metric \(h_\lambda\).  The~eigendistributions of \(P_\lambda\) are integrable and are determined by the vector field \(V_\lambda\) and the one-form \(\omega=d(\log R)\). The~polynomial structures are induced by the projector associated with the rank-one tensor \(\tilde g\) and the Hessian metric \(h_\lambda\).  
The construction is developed in arbitrary dimension, while the two-dimensional case is considered in~detail. 

The obtained structures depend on the choice of the Hessian metric \(h_\lambda\). In~particular, the~deformation parameter \(\lambda\) changes the corresponding projector, the~induced polynomial structures, and~their geometric~properties.

The construction in the paper shows that a rank-one tensor obtained from reciprocal cost geometry, together with a nondegenerate Hessian metric, determines a projector and the associated almost product, golden, and~metallic structures. In~this way, reciprocal cost geometry  leads to polynomial structures on Hessian manifolds. {\color{black} It is useful to summarize the geometric meaning of the construction.
The projector \(P_\lambda\) determines the distinguished cost direction
\(V_\lambda\), generated by the gradient of \(\log R\), and~splits the tangent
space into the one-dimensional distribution
\(\operatorname{im}(P_\lambda)=\operatorname{span}\{V_\lambda\}\) and its
\(h_\lambda\)-orthogonal complement
\(\ker(P_\lambda)=\ker\omega\). Since both eigendistributions are integrable,
they define two complementary foliations of the Hessian manifold, one along the
cost direction and one transverse to it. The~induced almost product, golden,
and \((p,q)\)-metallic structures preserve this decomposition and act with
different eigenvalues on the two foliations. Moreover, this decomposition is,
in general, not preserved by parallel transport with respect to either the
canonical flat connection or the Levi-Civita connection of \(h_\lambda\).
Thus, the~reciprocal cost function, the~associated Hessian metric, and~the
induced polynomial structures form a natural geometric construction on the Hessian
manifold. }

Further directions will include the study of the positive definite locus and curvature properties of \(h_\lambda\) and possible relations with Hessian~\cite{shima2007} and information geometry~\cite{amari}. {\color{black} It would  be of interest to study analogous constructions for other rank-one Hessian tensors and the corresponding induced polynomial structures.}

\vspace{6pt}
%%%%%%%%%%%%%%%%%%%%%%%%%%%%%%%%%%%%%%%%%%
\noindent {\bf Author Contributions:} {Conceptualization, J.W.; Methodology, J.W. and M.Z.; Software, J.W.; Validation, J.W. and M.Z.; Formal Analysis, M.Z. and J.W.; Investigation, J.W. and M.Z.; Resources,
J.W.; Writing---Original Draft Preparation, J.W.; Writing---Review and Editing, M.Z. and J.W.; Funding Acquisition, J.W. All authors have read and agreed to the published version of the~manuscript.}

%%%%%%%%%%%%%%%%%%%%%%%%%%%%%%%%%%%%%%%%%%

\vspace{6pt}
%%%%%%%%%%%%%%%%%%%%%%%%%%%%%%%%%%%%%%%%%%
\noindent
{\color{black}{\bf Acknowledgments:}{ {The} %MDPI: Please ensure that all individuals included in this section have consented to the acknowledgement. %MZ Yes, it is okay
 authors thank Elshad Allahyarov, Anil Thapa, Philip Beltracchi, and Megan
Simons for their valuable comments on an earlier version of the paper. }} %MDPI: Please ensure that all individuals included in this section have consented to the acknowledgement.%MZ it is okay
\end{document}